\documentclass{amsart}

\usepackage{amsmath,amssymb,amsthm,a4wide}
\usepackage{pgfplots}
\usepackage{graphicx,tikz}

\newtheorem*{thm}{Theorem}

\newtheorem{corollary}{Corollary}

\newtheorem*{lem}{Lemma}

\theoremstyle{definition}

\theoremstyle{remark}

\newcommand{\vol}{\operatorname{vol}}

\begin{document}

\title[]{localized quantitative criteria for equidistribution}

\author[]{Stefan Steinerberger}
\address{Department of Mathematics, University of Washington, Seattle, WA 98195, USA}
\email{steinerb@uw.edu}

\begin{abstract}  Let $(x_n)_{n=1}^{\infty}$ be a sequence on the torus $\mathbb{T}$ (normalized to length 1). We show that if
there exists a sequence of positive real numbers $(t_n)_{n=1}^{\infty}$ converging to 0 such that
 $$ \lim_{N \rightarrow \infty}{ \frac{1}{N^2} \sum_{m,n = 1}^{N}{  \frac{1}{\sqrt{t_N}} \exp{\left(- \frac{1}{t_N} (x_m - x_n)^2 \right)}} } 
 = \sqrt{\pi},$$
then $(x_n)_{n=1}^{\infty}$ is uniformly distributed. This is especially interesting when $t_N$ is close to $N^{-2}$ since the size of the sum is then mostly determined by local gaps at scale $\sim N^{-1}$. A similar argument can then be used to show equidistribution of sequences
with Poissonian pair correlation, which recovers a recent result of  Aistleitner, Lachmann \& Pausinger and Grepstad \& Larcher. The general
form of the result is proven on arbitrary compact manifolds $(M,g)$ where the role of the exponential function is played by
the heat kernel $e^{t\Delta}$: for all $x_1, \dots, x_N \in M$ and all $t>0$
$$ \frac{1}{N^2}\sum_{m,n=1}^{N}{[e^{t\Delta}\delta_{x_m}](x_n)} \geq \frac{1}{\vol(M)}$$
and equality is attained as $N \rightarrow \infty$ if and only if $(x_n)_{n=1}^{\infty}$ equidistributes.
\end{abstract}

\maketitle

\section{Introduction}

\subsection{Introduction.} 
 Let $(x_n)_{n=1}^{\infty}$ be a sequence on $[0,1]$. A naturally associated object of interest is the behavior of gaps on a local scale. If the sequence is comprised of independently and uniformly distributed random variables, then 
 $$ \lim_{N \rightarrow \infty}{ \frac{1}{N} \# \left\{ 1 \leq m \neq n \leq N: |x_m - x_n| \leq \frac{s}{N} \right\}} = 2s \qquad \mbox{almost surely.}$$ 
 Whenever a deterministic sequence $(x_n)_{n=1}^{\infty}$ has the same property, we say it has Poissonian pair correlation; this notion of pseudorandomness has been intensively investigated, see e.g. \cite{heath, rud1, rud2, rud3}.
Recently,  Aistleitner, Lachmann \& Pausinger \cite{aist} and Grepstad \& Larcher \cite{grep} independently established that sequences with Poissonian pair correlation are uniformly distributed on $[0,1]$.
 \begin{thm}[Aistleitner-Lachmann-Pausinger \cite{aist}, Grepstad-Larcher \cite{grep}]  Let $(x_n)_{n=1}^{\infty}$ be a sequence on $[0,1]$ and assume that for all $s>0$
  $$ \lim_{N \rightarrow \infty}{ \frac{1}{N} \# \left\{ 1 \leq m \neq n \leq N: |x_m - x_n| \leq \frac{s}{N} \right\}} = 2s \qquad \mbox{a.s.},$$ 
  then the sequence is uniformly distributed.
 \end{thm}
 An intuitive explanation is that any type of clustering produces many pairs $(x_m,x_n)$ which are close to each other -- the two available proofs are very different; the proof in \cite{grep} also implies a quantitative estimate on discrepancy. The purpose of this paper is to embed this result into an entire family of criteria that imply uniform distribution -- this family of criteria strenghten the result cited above
and are applicable on general compact manifolds.

\subsection{Qualitative Results on the Torus.} We start by formulating our result in the special case of the one-dimensional torus $\mathbb{T}$ (normalized to have length 1): let
  $$ \theta_t(x) = 1 + 2 \sum_{n=1}^{\infty}{e^{-4 \pi^2 n^2 t}\cos{(2 \pi n x)}} \qquad \mbox{denote the Jacobi $\theta-$function.}$$
 We observe that $\theta_t(x) \geq 0$, $\theta_t(x)=\theta_t(1-x)$, 
  $$ \int_{0}^{1}{\theta_t(x)dx} = 1 \qquad \mbox{and} \qquad   \theta_{t}(x) \sim \begin{cases} 1/\sqrt{t} \qquad &\mbox{for}~ |x| \lesssim \sqrt{t} \\ 0 \qquad &\mbox{otherwise.} \end{cases}$$
It looks roughly like a Gaussian centered at 0 with standard deviation $\sim t^{1/2}$ (the profile indeed converges to that of a Gaussian as $t \rightarrow 0$). Note that the property $\theta_t(x)=\theta_t(1-x)$ implies that we never have to distinguish
between points on the unit interval $[0,1]$ and points on the torus $\mathbb{T}$ of length 1.
Our main result, when applied to the one-dimensional torus $\mathbb{T}$, can be stated as follows.
 \begin{corollary}
 Let $(x_n)_{n=1}^{\infty}$ be a sequence on $\mathbb{T}$. If there exists a sequence $(t_n)_{n=1}^{\infty}$ of positive and bounded real numbers such that 
 $$ \lim_{N \rightarrow \infty}{ \frac{1}{N^2} \sum_{m,n = 1}^{N}{\theta_{t_N}(x_n - x_m)}} = 1,$$
 then $(x_n)_{n=1}^{\infty}$ is uniformly distributed on $\mathbb{T}$. 
 \end{corollary}
If the sequence $(t_n)_{n=1}^{\infty}$ additionally converges to 0, then we can simplify the criterion by replacing the Jacobi $\theta-$function
with a suitable scaled Gaussian: indeed, if $t_N \rightarrow 0$, then
 $$ \lim_{N \rightarrow \infty}{ \frac{1}{N^2} \sum_{m,n = 1}^{N}{  \frac{1}{\sqrt{t_N}} \exp{\left(- \frac{1}{t_N} (x_m - x_n)^2 \right)}} } 
 = \sqrt{\pi}$$
immediately implies
$$  \lim_{N \rightarrow \infty}{ \frac{1}{N^2} \sum_{m,n = 1}^{N}{\theta_{t_N}(x_n - x_m)}} = 1.$$
\begin{center}
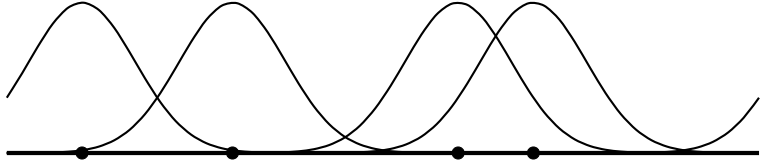
\begin{figure}[h!]
\begin{tikzpicture}
\draw [ultra thick] (0,0) -- (10,0);
\filldraw (1,0) circle (0.08cm);
\filldraw (3,0) circle (0.08cm);
\filldraw (6,0) circle (0.08cm);
\filldraw (7,0) circle (0.08cm);
 \draw[scale=1,thick, samples=50,domain=0:10,smooth,variable=\x]  plot ({\x},{2*exp(-(\x-1)*(\x-1))});
 \draw[scale=1,thick, samples=50,domain=0:10,smooth,variable=\x]  plot ({\x},{2*exp(-(\x-11)*(\x-11))});
 \draw[scale=1,thick, samples=50,domain=0:10,smooth,variable=\x]  plot ({\x},{2*exp(-(\x-3)*(\x-3))});
 \draw[scale=1,thick, samples=50,domain=0:10,smooth,variable=\x]  plot ({\x},{2*exp(-(\x-6)*(\x-6))});
 \draw[scale=1,thick, samples=50,domain=0:10,smooth,variable=\x]  plot ({\x},{2*exp(-(\x-7)*(\x-7))});
\end{tikzpicture}
\caption{Highly localized Gaussians evaluated at neighboring points.}
\end{figure}
\end{center}
\vspace{-10pt}
It is not difficult to see (and will be established as part of our argument) that for any $0 < s < t$ and any $x_1, x_2, \dots, x_N \in \mathbb{T}$
$$ \frac{1}{N^2} \sum_{m,n = 1}^{N}{\theta_{s}(x_n - x_m)} \geq  \frac{1}{N^2} \sum_{m,n = 1}^{N}{\theta_{t}(x_n - x_m)} \geq 1.$$
This shows that the criterion becomes more restrictive if the sequence of scales $(t_n)_{n=1}^{\infty}$ is made smaller. Conversely,
if that sequence is taken to be the constant sequence, $t_n = 1$ for all $n\in \mathbb{N}$, the criterion becomes sharp and characterizes uniform distribution.

\begin{corollary} A sequence $(x_n)_{n=1}^{\infty}$ is uniformly distributed on $[0,1]$ or $\mathbb{T}$ if and only if
 $$ \lim_{N \rightarrow \infty}{ \frac{1}{N^2} \sum_{m,n = 1}^{N}{\theta_{1}(x_n - x_m)}} = 1.$$
\end{corollary}
There is nothing special about $t=1$ and the result holds if $\theta_1$ is replaced by $\theta_c$ for any fixed $c > 0$.
 Our proof gives a little bit more information and shows
that we could replace $\theta_1$ by any function $\phi \in C^{\infty}(\mathbb{T})$ satisfying
$$ \int_{\mathbb{T}}{\phi(x) dx} = 1, ~\phi(x) = \phi(1-x) \qquad \mbox{and} \qquad  \int_{\mathbb{T}}{\phi(x) e^{2 \pi i k x} dx} > 0 \quad \mbox{for all}~k \in \mathbb{N}.$$
This result is related in spirit to the classical Bochner-Herglotz theorem and variants exist on other
topological groups (see e.g. \cite{kui}). We emphasize that our Corollary, using the Jacobi $\theta-$function, is a consequence of a 
result on general compact manifolds that does not assume any type of group structure.

\subsection{Application to pair correlation.} The result cited above states that if for all $s>0$
  $$ \lim_{N \rightarrow \infty}{ \frac{1}{N} \# \left\{ 1 \leq m \neq n \leq N: |x_m - x_n| \leq \frac{s}{N} \right\}} = 2s,$$ 
  then the sequence is uniformly distributed.
A natural question is whether it is truly necessary to require the limit relation to hold for \textit{all} $s>0$. 
We can use our criterion from above to show that it suffices to know it for all $s \in \mathbb{N}$ (slightly sharper results could be obtained but this is not the focus of this paper).
This should still be far from optimal and sharper criteria could be of interest. 

 \begin{corollary}[Pair correlation]  Let $(x_n)_{n=1}^{\infty}$ be a sequence on $[0,1]$ and assume that for all $s \in \mathbb{N}$
  $$ \lim_{N \rightarrow \infty}{ \frac{1}{N} \# \left\{ 1 \leq m \neq n \leq N: |x_m - x_n| \leq \frac{s}{N} \right\}} = 2s,$$ 
  then the sequence is uniformly distributed.
 \end{corollary}
The criterion cannot be directly applied to use pair correlation: for that we would be required to work on
the scale $t_N \sim N^{-2}$ and it is not too difficult to see that this is where the criterion has to stop working because the diagonal terms are already too large
 $$  \frac{1}{N^2} \sum_{m,n = 1}^{N}{\theta_{t_N}(x_n - x_m)} \geq  \frac{1}{N^2} \sum_{n = 1}^{N}{\theta_{t_N}(x_n - x_n)} = \frac{\theta_{t_N}(0)}{N} \sim \frac{1}{N t_N^{1/2}} \gtrsim 1.$$
However, it is fairly easy to see that it is possible to make a slight adaption to the argument and we will describe this adaption at the end of the paper. Many variants are conceivable, we prove the following natural generalization.
 \begin{corollary}[Weak Pair correlation]  Let $(x_n)_{n=1}^{\infty}$ be a sequence on $[0,1]$, let $0 < \alpha < 1$ and assume that for all $s > 0$
  $$ \lim_{N \rightarrow \infty}{ \frac{1}{N^{2-\alpha}} \# \left\{ 1 \leq m \neq n \leq N: |x_m - x_n| \leq \frac{s}{N^{\alpha}} \right\}} = 2s \qquad \mbox{a.s.},$$ 
  then the sequence is uniformly distributed.
 \end{corollary}
We emphasize that this result is more widely applicable since the requirement of weak pair correlation is less stringent. Consider, for example, a sequence $(x_n)_{n=1}^{\infty}$
obtained by taking $x_{2n-1}$ to be i.i.d. uniformly chosen random variables on $[0,1]$ and $x_{2n} = x_{2n-1}$. This sequence does not have Poissonian pair correlation
but satisfies the notion of weak pair correlation for all $0<\alpha<1$ a.s.
The argument naturally generalizes to other geometric spaces and can be roughly summarized as saying that whenever
$$ \#\left\{ 1 \leq m, n \leq N: \|x_m -x_n \| \leq s \cdot t_N \right\} \qquad \mbox{for some} \quad t_N \rightarrow 0 ~\mbox{and all}~s>0$$
behaves as it would for Poissonian random variables, then the sequence is uniformly distributed since criteria of this type can
be used to determine the validity of the limit relation in our criterion.

\subsection{A quantitative result.} Our method is flexible enough to allow for the derivation of quantitative results. We only discuss the simplest case,
the method applies to fairly general discrepancy systems on compact manifolds. Let now $x_1, x_2, \dots, x_N \in \mathbb{T}$. Discrepancy is defined
 as the maximum deviation of uniform and empirical distribution on the set of all intervals $J \subset \mathbb{T}$
 $$D_N(\left\{x_1, \dots, x_N\right\}) = \sup_{J \subset \mathbb{T}}{\left| \frac{\#\left\{1 \leq i \leq N: x_i \in J\right\}}{N} - |J| \right| }.$$
 It is easy to see that $D_N$ tends to $0$ as $N \rightarrow \infty$ if and only if the sequence is uniformly distributed. We recall
that for all $x_1, \dots, x_N \in \mathbb{T}$
$$ \frac{1}{N^2} \sum_{m,n = 1}^{N}{\theta_{t}(x_n - x_m)} \geq 1 \qquad \mbox{and decreases monotonically in}~t.$$
 \begin{corollary}[Discrepancy bound] There exists a universal constant $c>0$ such that for any $x_1, x_2, \dots, x_N \in \mathbb{T}$
  $$ D_N^2 \leq c \left( \frac{1}{N^2} \sum_{m,n = 1}^{N}{\theta_{-\frac{D_N^2}{c \log{D_N}}}(x_n - x_m)}  -  1 \right)$$
 \end{corollary}
We believe the result to be somewhat amusing but do not know whether it can be useful in a more general context.
We note that the $\theta-$function operates on spatial scale $\sim D_N \left(\log{D_N}\right)^{-1/2}$. 
Similar results can be obtained on general manifolds using the same argument. We point out a connection to
crystallization problems: given $N$ points on a manifold interacting via a nonlocal energy, minimizing configuration
often arrange themselves into periodic structures (we refer to \cite{blanc} for an introduction and to \cite{ber,mon}
for results involving the Jacobi $\theta-$function).

\subsection{The general result.}  Let $(M,g)$ be a smooth compact manifold. We use  $e^{t\Delta}$ to denote the heat kernel, i.e. the semigroup that allows to solve the heat equation
$$
(\partial_t - \Delta_g) e^{t\Delta} u_0 = 0 \qquad \mbox{on}~M \times [0,\infty]
$$
 and will apply it mostly to Dirac $\delta$ functions located on the manifold. Note that
classical short-time asymptotics show
$$ \left[ e^{t\Delta} \delta_x\right](y) \sim t^{-d/2} \exp \left( -\frac{d(x,y)^2}{4t} \right),$$
where $d(x,y)$ is the geodesic distance: we are therefore, for $t$ sufficiently small, essentially dealing with Gaussians centered at $x$.
We will prove the general inequality
$$  \frac{1}{N^2} \sum_{m,n = 1}^{N}{ \left(e^{t\Delta}\delta_{x_m}\right)(x_n)   } \geq \frac{1}{\mbox{vol}(M)}$$
and show that asymptotic sharpness characterizes uniform distribution of $(x_n)_{n=1}^{\infty}$.
\vspace{-10pt}
\begin{center}
\begin{figure}[h!]
\begin{tikzpicture}[scale=0.8]
    \draw[thick, smooth cycle,tension=.7] plot coordinates{(-1,0) (0.5,2) (2,2) (4,3) (4.5,0)};
    \coordinate (A) at (1,1);
    \draw[thick] (A) arc(140:40:1) (A) arc(-140:-20:1) (A) arc(-140:-160:1);
\filldraw (4,1) ellipse (0.05cm and 0.05cm);
\draw[rotate around={45:(4,1)}, thick] (4,1) ellipse (1*0.28cm and 1*0.18cm);
\draw[rotate around={45:(4,1)}] (4,1) ellipse (2*0.28cm and 2*0.18cm);
\draw[rotate around={45:(4,1)}, dashed] (4,1) ellipse (3*0.28cm and 3*0.18cm);
\end{tikzpicture}
\caption{$\left[ e^{t\Delta} \delta_x\right](y)$ behaves like a Gaussian centered at $y$ and scale $\sim \sqrt{t}$.}
\end{figure}
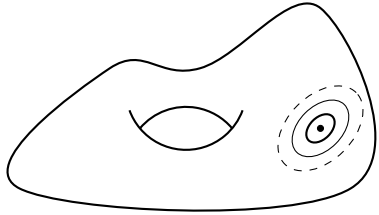
\end{center}

\begin{thm}  Let $(M,g)$ be a smooth compact manifold and let $(x_n)_{n=1}^{\infty}$ be a sequence on $M$. If there exists a bounded sequence of times $0 < t_N \leq C$ such that
$$ \lim_{N \rightarrow \infty}{  \frac{1}{N^2} \sum_{m,n = 1}^{N}{ \left(e^{t_N\Delta}\delta_{x_m}\right)(x_n)   }} = \frac{1}{\vol(M)},$$
then the sequence is uniformly distributed on $(M,g)$. Moreover, if $t_N = c > 0$ is constant, then this limit relation holds if and only if $(x_n)_{n=1}^{\infty}$ is uniformly distributed.
\end{thm}
It is immediately clear that on special manifolds such as $\mathbb{T}^d$ or $\mathbb{S}^{d-1}$ (where explicit formulas for the heat kernel are available), the result could be simplified and put into a similar form as the corollaries above on $[0,1]$ or $\mathbb{T}$. It is also not difficult to see that the condition in the Theorem can never be satisfied if $t_N$ decays faster than $N^{-2/d}$ since
$$ \frac{1}{N^2} \sum_{m,n = 1}^{N}{ \left(e^{t_N\Delta}\delta_{x_m}\right)(x_n)   } \geq \frac{1}{N^2} \sum_{m = 1}^{N}{ \left(e^{t_N\Delta}\delta_{x_m}\right)(x_m)   } \gtrsim \frac{1}{N} t_N^{-d/2},$$
which is $\sim 1$ for $t_N \sim N^{-2/d}$. We quickly sketch what happens if we apply the result on the torus: the heat kernel has the explicit form
$$ [e^{t \Delta}(\delta_x)](y) = \theta_t(x-y).$$
The comparison between the Jacobi $\theta-$function $\theta_t(x)$ for $t$ small comes from the asymptotic estimate
$$ \theta_t(x) \sim \frac{1}{\sqrt{4 \pi t}} \exp \left( -\frac{x^2}{4t} \right).$$
It is easy to see the incurred error is small (indeed, many orders smaller than would be required).

\section{Proof of the Main Theorem}

\subsection{Warming up.} Before going into details, we give a very simple argument for the monotonicity formula in the simplest case and
prove that for any $0 < s < t$ and any $x_1, x_2, \dots, x_N \in \mathbb{T}$
$$ \frac{1}{N^2} \sum_{m,n = 1}^{N}{\theta_{s}(x_n - x_m)} \geq  \frac{1}{N^2} \sum_{m,n = 1}^{N}{\theta_{t}(x_n - x_m)} \geq 1.$$ 
The proof is straightforward: first, we rewrite the expression as
$$  \frac{1}{N^2} \sum_{m,n = 1}^{N}{\theta_{s}(x_n - x_m)} = \int_{\mathbb{T}}{  \theta_s * \left(\frac{1}{N}\sum_{n=1}^{N}{\delta_{x_n}} \right) \overline{\left( \frac{1}{N}\sum_{n=1}^{N}{\delta_{x_n}} \right)} dx}.$$
The Plancherel identity yields
$$ \int_{\mathbb{T}}{  \theta_s * \left(\frac{1}{N}\sum_{n=1}^{N}{\delta_{x_n}} \right) \overline{ \left( \frac{1}{N}\sum_{n=1}^{N}{\delta_{x_n}} \right) } dx} = \sum_{\ell \in \mathbb{Z}}{ \widehat{ \theta_s  * \left(\frac{1}{N}\sum_{n=1}^{N}{\delta_{x_n}} \right)}(\ell) \overline{ \widehat{ \left( \frac{1}{N}\sum_{n=1}^{N}{\delta_{x_n}} \right)}(\ell)} }.$$
We note that
$$  \widehat{   \theta_t  * \left(\frac{1}{N}\sum_{n=1}^{N}{\delta_{x_n}} \right) }(\ell) = \widehat{\theta_t}(\ell)  \widehat{ \left(\frac{1}{N}\sum_{n=1}^{N}{\delta_{x_n}} \right)}(\ell) \qquad \mbox{and} \qquad  \widehat \theta_t = \sum_{\ell \in \mathbb{Z}}{e^{-4 \pi^2 \ell^2 t} e^{2\pi i \ell x}}.$$
Finally, an expansion into Fourier series
$$  \frac{1}{N}\sum_{n=1}^{N}{\delta_{x_n}}  = \sum_{\ell \in \mathbb{Z}}{a_{\ell} e^{2\pi i \ell x}}$$
with $a_{\ell} = a_{-\ell}$ and $a_0 = 1$ allows us to write
$$ \int_{\mathbb{T}}{  \theta_t * \left(\frac{1}{N}\sum_{n=1}^{N}{\delta_{x_n}} \right) \overline{ \left( \frac{1}{N}\sum_{n=1}^{N}{\delta_{x_n}} \right) } dx } = \sum_{\ell \in \mathbb{Z}}{e^{-4 \pi^2 \ell^2 t}  |a_{\ell}|^2} = 1 + \sum_{\ell \in \mathbb{Z}\atop \ell \neq 0}{e^{-4 \pi^2 \ell^2 t}  |a_{\ell}|^2} ,$$
which is monotonically decreasing in $t$. On general manifolds, we will repeat the argument with the Fourier basis of $L^2$ being replaced by eigenfunctions of $-\Delta_g$. The semigroup
properties of the heat kernel serve as a substitute for the behavior of convolution under Fourier transform.

\subsection{Structure of the Argument.} We quickly outline the overall structure of the argument and will then divide the proof accordingly.
The proof has five steps.

\begin{enumerate}
\item  We will start by showing that
$$\frac{1}{N^2} \sum_{m,n = 1}^{N}{ \left(e^{t\Delta}\delta_{x_m}\right)(x_n)   }  \qquad  \mbox{is monotonically decreasing in}~t.$$
Since we are dealing with a bounded sequence of times $0 < t_N \leq C$, monotonicity implies that it suffices to prove the main result only for $t_n = C$.

\item If the sequence is not uniformly distributed, there exists a ball $B$ and $\varepsilon > 0$ such that for infinitely many $N$ there are $(|B|+\varepsilon)N/\vol(M)$ 
out of the first $N$ elements contained in the ball $B$. We then consider, for a sufficiently small but fixed time $\delta = \delta_{B, \varepsilon} > 0$, the function
$$  e^{\delta \Delta} \frac{1}{N} \sum_{m = 1}^{N}{ \delta_{x_m}}$$
and prove that the average value of the function in a small neighborhood of $B$ is bigger than $\vol(M)^{-1}+c_{\varepsilon, \delta, B} $ for some fixed $c_{\varepsilon, \delta, B} > 0$ and infinitely many $N \in \mathbb{N}$.
\item The Cauchy-Schwarz inequality then implies
$$ \left\| e^{\delta \Delta} \frac{1}{N} \sum_{m = 1}^{N}{ \delta_{x_m}} - \frac{1}{\vol(M)}   \right\|_{L^2(M)} \geq c^*_{\varepsilon, \delta, B} \qquad ~\mbox{for infinitely many}~N\in \mathbb{N}.$$
\item We use the spectral theorem, the eigenfunctions $(\phi_k)_{k=1}^{\infty}$ of the Laplace-Beltrami operator $-\Delta_g$ and inequalities related to the compactness $e^{\delta \Delta}:L^1(M) \rightarrow H^s(M)$ to conclude that there exists a constant $N_0 \in \mathbb{N}$, depending only on $B, \varepsilon, \delta$ such that for all $N$ and all $x_1, \dots, x_N \in M$
$$\sum_{k \leq N_0}{ \left| \left\langle e^{\delta \Delta} \frac{1}{N} \sum_{m = 1}^{N}{ \delta_{x_m}} - \frac{1}{\vol(M)} , \phi_k \right\rangle \right|^2} \geq \frac{1}{2} \left\| e^{\delta \Delta} \frac{1}{N} \sum_{m = 1}^{N}{ \delta_{x_m}} - \frac{1}{\vol(M)}   \right\|^2_{L^2(M)}.$$
 Combining these last two steps implies the existence of infinitely many $N \in \mathbb{N}$ such that
$$ \sum_{k \leq N_0}{ \left| \left\langle e^{\delta \Delta} \frac{1}{N} \sum_{m = 1}^{N}{ \delta_{x_m}} - \frac{1}{\vol(M)} , \phi_k \right\rangle \right|^2} \geq \frac{1}{2}\left(c^{*}_{\varepsilon, \delta, B}\right)^2 > 0$$
 and, finally, we can use this to show that
$$\sum_{k \leq N_0}{ \left| \left\langle e^{C \Delta}  \frac{1}{N}\sum_{m = 1}^{N}{ \delta_{x_m}} - \frac{1}{\vol(M)} , \phi_k \right\rangle \right|^2} \geq c^{**}_{\varepsilon, \delta, B, N_0,C} > 0 \quad \mbox{for infinitely many}~N\in \mathbb{N}.$$
\item We conclude by arguing that
$$
\left\langle  e^{2C \Delta}  \frac{1}{N}\sum_{m = 1}^{N}{ \delta_{x_m}},  \frac{1}{N}\sum_{m = 1}^{N}{ \delta_{x_m}} \right\rangle  \geq  \frac{1}{\vol(M)} + \sum_{k \leq N_0}^{}{ \left| \left\langle  e^{C \Delta}  \frac{1}{N}\sum_{m = 1}^{N}{ \delta_{x_m}}, \phi_k \right\rangle \right|^2}
$$
from which the result then follows when using that
$$ \left\langle  e^{C \Delta}  \frac{1}{N}\sum_{m = 1}^{N}{ \delta_{x_m}} - \frac{1}{\vol(M)}, \phi_k \right\rangle = \left\langle  e^{C \Delta}  \frac{1}{N}\sum_{m = 1}^{N}{ \delta_{x_m}}, \phi_k \right\rangle  \qquad \mbox{for all}~k \geq 1$$
because these eigenfunctions are orthogonal to $\phi_0(x) = \vol(M)^{-1/2}$.
\end{enumerate}

\subsection{Proof of the Theorem.}

\begin{proof}
Let $(M,g)$ be given and let us consider the $L^2-$normalized Laplacian eigenfunctions
$$ -\Delta_g \phi_n = \lambda_n \phi_n$$
as a basis of $L^2(M)$. Observe that $\lambda_0 = 0$ and $\phi_0 = \vol(M)^{-1/2}$ is a constant function.\\
\textit{Step 1.} We can rewrite the expression as
$$\frac{1}{N^2} \sum_{m,n = 1}^{N}{ \left(e^{t_N\Delta}\delta_{x_m}\right)(x_n)   }  =
 \left\langle e^{t_N\Delta}  \left( \frac{1}{N} \sum_{m=1}^{N}{\delta_{x_m}} \right), \left( \frac{1}{N} \sum_{m=1}^{N}{\delta_{x_m}}\right) \right\rangle.$$
This particular algebraic structure behaves well under the heat flow: for any $f \in C^{\infty}(M)$, we can write
$$ f = \sum_{k =0}^{\infty}{ \left\langle f, \phi_k \right\rangle \phi_k} \qquad \mbox{and} \qquad  e^{t\Delta}f = \sum_{k =0}^{\infty}{ e^{-\lambda_k t} \left\langle f, \phi_k \right\rangle \phi_k}  $$
and thus
$$ \left\langle e^{t \Delta} f, f \right\rangle = \sum_{k =0}^{\infty}{ e^{-\lambda_k t} \left\langle f, \phi_k \right\rangle^2}.$$
This quantity is obviously monotonically decreasing in $t$. Note that
$$\lim_{t \rightarrow \infty}{ \left\langle e^{t \Delta} f, f \right\rangle} =  \left\langle f, \phi_0 \right\rangle^2 =  \left\langle f, \frac{1}{\vol(M)^{1/2}} \right\rangle^2 = \frac{1}{\vol(M)^{}} \left(  \int_{M}{ f dg}\right)^2,$$
which immediately implies, using a density argument, that for all $t > 0$
$$ \left\langle e^{t\Delta}  \left( \frac{1}{N} \sum_{m=1}^{N}{\delta_{x_m}} \right), \left( \frac{1}{N} \sum_{m=1}^{N}{\delta_{x_m}}\right) \right\rangle \geq \frac{1}{\vol(M)}.$$

\textit{Step 2.} Let us now assume that the sequence $(x_n)_{n=1}^{\infty}$ is not uniformly distributed but that nonetheless
$$ \lim_{N \rightarrow \infty}{  \frac{1}{N^2} \sum_{m,n = 1}^{N}{ \left(e^{t_N\Delta}\delta_{x_m}\right)(x_n)   }} = \frac{1}{\vol(M)}.$$
The monotonicity of the expression under the heat flow implies that also
$$ \lim_{N \rightarrow \infty}{  \frac{1}{N^2} \sum_{m,n = 1}^{N}{ \left(e^{C\Delta}\delta_{x_m}\right)(x_n)   }} = \frac{1}{\vol(M)},$$
where $C$ is the uniform upper bound on the sequence of times $t_N$. Not being uniformly distributed means there exists a geodesic ball $B \subset M$ and $\varepsilon_0 > 0$ such that
 $$ \left| \frac{  \# \left\{1 \leq m \leq N: x_m \in B\right\}}{N} - \frac{|B|}{\vol(M)} \right| \geq \varepsilon_0 \qquad \mbox{for infinitely many}~N.$$
We shall assume that
 $$  \frac{  \# \left\{1 \leq m \leq N: x_m \in B\right\}}{N} -  \frac{|B|}{\vol(M)} \geq \varepsilon_0 \qquad \mbox{for infinitely many}~N$$
because the other case implies the same estimate for another ball $B' \subseteq M \setminus B$ (possibly with a different value of $\varepsilon_0$). 
We may assume without loss of generality (by possibly making $\varepsilon_0$ smaller) that $|B| \leq 1/2$.
Let $B_{\delta}$ denote the $\delta-$neighborhood of $B$ and let $\delta > 0$
be chosen so small that
$$ \frac{|B_{\delta}|}{|B|} \leq  1+ \frac{\varepsilon_0}{100}$$
and let $t_0 > 0$ be chosen so small that
$$ \inf_{z \in B} \int_{B_{\delta}}{ \left[e^{t_0 \Delta} \delta_z\right](x) dx} \geq 1-\frac{\varepsilon_0}{100}.$$
These two facts imply that for infinitely many $N$
$$  \int_{B_{\delta}}{ \left[e^{t_0 \Delta} \frac{1}{N} \sum_{m=1}^{N}{\delta_{x_m}} \right](x) dx} \geq \left(1-\frac{\varepsilon_0}{100} \right) (|B| + \varepsilon_0) \frac{1}{\vol(M)}.$$
This means that the average value satisfies
\begin{align*}
 \frac{1}{|B_{\delta}|}\int_{B_{\delta}}{ \left[e^{t_0 \Delta} \frac{1}{N} \sum_{m=1}^{N}{\delta_{x_m}} \right](x) dx} &\geq \left(1-\frac{\varepsilon_0}{100} \right) \frac{|B| + \varepsilon_0}{|B_{\delta}|}\frac{1}{\vol(M)} \\
&\geq \left(1-\frac{\varepsilon_0}{100} \right) \left( \frac{|B|}{|B_{\delta}|} + \frac{\varepsilon_0}{|B_{\delta}|} \right)\frac{1}{\vol(M)} \\
&\geq \left(1-\frac{\varepsilon_0}{100} \right) \left( 1 - \frac{\varepsilon_0}{100}+ \varepsilon_0 \right)\frac{1}{\vol(M)}\\
&\geq \left(1 + \frac{98 \varepsilon_0}{100} - \frac{99 \varepsilon_0^2}{10000}\right)\frac{1}{\vol(M)} > \frac{1}{\vol(M)}
\end{align*}
\textit{Step 3.} The Cauchy-Schwarz inequality implies, for general functions $f$  that
\begin{align*} |B_{\delta}| \left|  \frac{1}{|B_{\delta}|}  \int_{B_{\delta}}{f d x} - \frac{1}{\vol(M)} \right|  &=  \left| \int_{B_{\delta}}{\left(f-\frac{1}{\vol(M)}\right) ~dx} \right| \\
&\leq  \left(\int_{B_{\delta}}{\left(f-\frac{1}{\vol(M)}\right)^2 d x}\right)^{1/2} |B_{\delta}|^{1/2}
\end{align*}
and therefore there exists a constant $\varepsilon_1 > 0$ (depending only on $\varepsilon_0$ and $|B_{\delta}|$) such that
for infinitely many $N \in \mathbb{N}$
$$ \left\|   \left[e^{t_0 \Delta} \frac{1}{N} \sum_{m=1}^{N}{\delta_{x_m}} \right](x) - \frac{1}{\vol(M)} \right\|_{L^2(M)} \geq \varepsilon_1.$$

\textit{Step 4.} We will now prove that for any $t_0 > 0$, any $N \in \mathbb{N}$ and any set of points $x_0, x_1, \dots, x_N$
$$  \left\|  \nabla  e^{t_0 \Delta} \frac{1}{N} \sum_{m=1}^{N}{\delta_{x_m}}  \right\|_{L^2(M)} \lesssim_{t_0, (M,g)} 1,$$
where $ \lesssim_{t_0, (M,g)}$ denotes the existence of an implicit constant depending only on $t_0$ and $(M,g)$.
for some implicit constant that is both independent of $N$ and the actual set of points. It is easy to see that
\begin{align*}
 \left\|  \nabla e^{t_0 \Delta} \frac{1}{N} \sum_{m=1}^{N}{\delta_{x_m}}  \right\|_{L^{\infty}(M)} &=  \left\|  \frac{1}{N} \sum_{m=1}^{N}{  \nabla e^{t_0 \Delta} \delta_{x_m}}  \right\|_{L^{\infty}(M)} \\
&\leq  \frac{1}{N} \sum_{m=1}^{N}{  \left\|    \nabla e^{t_0 \Delta} \delta_{x_m} \right\|_{L^{\infty}(M)} }\\
&\leq \sup_{x \in M} {  \left\|    \nabla e^{t_0 \Delta} \delta_{x} \right\|_{L^{\infty}(M)}} \lesssim_{(M,g), t_0} 1,
\end{align*}
which follows from the regularity of the Green's function. By the same token
\begin{align*}
 1 =  \left\|  e^{t_0 \Delta} \frac{1}{N} \sum_{m=1}^{N}{\delta_{x_m}}  \right\|_{L^{1}(M)} &\leq \vol(M)^{1/2}  \left\|  e^{t_0 \Delta} \frac{1}{N} \sum_{m=1}^{N}{\delta_{x_m}}  \right\|_{L^{2}(M)}\\
&\lesssim_{(M,g)}   \left\|\frac{1}{N} \sum_{m=1}^{N}{   e^{t_0 \Delta}  \delta_{x_m}}  \right\|_{L^{2}(M)}\\
&\leq  \sup_{x \in M} {  \left\|  e^{t_0 \Delta} \delta_{x} \right\|_{L^{2}(M)}} \lesssim_{(M,g), t_0} 1.
 \end{align*}
We can now use the spectral theorem to write
$$ 1 \gtrsim_{(M,g), t_0}  \left\|  \nabla  e^{t_0 \Delta} \frac{1}{N} \sum_{m=1}^{N}{\delta_{x_m}}  \right\|^2_{L^2(M)} = \sum_{k=0}^{\infty}{\lambda_k \left|\left\langle e^{t_0 \Delta} \frac{1}{N} \sum_{m=1}^{N}{\delta_{x_m}}, \phi_k \right\rangle \right|^2 }.$$
However, at the same time, the eigenvalues of the Laplace-Beltrami operator are monotonically increasing and unbounded
$$0 < \lambda_1 < \lambda_2 \leq \dots \rightarrow \infty.$$ 
Weyl's law would give the asymptotic growth but that is not necessary. Recall that we have that for infinitely many $N \in \mathbb{N}$
$$ \left\|   \left[e^{t_0 \Delta} \frac{1}{N} \sum_{m=1}^{N}{\delta_{x_m}} \right](x) - \frac{1}{\vol(M)} \right\|_{L^2(M)} \geq \varepsilon_1.$$

 We can now argue that for $N_1 \geq 1$
\begin{align*}
\varepsilon_1^2 &\leq  \left\|   \left[e^{t_0 \Delta} \frac{1}{N} \sum_{m=1}^{N}{\delta_{x_m}} \right](x) - \frac{1}{\vol(M)} \right\|^2_{L^2(M)} \\
&=  \sum_{0 < \lambda_k \leq N_1}^{}{\left|\left\langle e^{t_0 \Delta} \frac{1}{N} \sum_{m=1}^{N}{\delta_{x_m}}, \phi_k \right\rangle \right|^2 }  +  \sum_{\lambda_k > N_1}^{}{\left|\left\langle e^{t_0 \Delta} \frac{1}{N} \sum_{m=1}^{N}{\delta_{x_m}}, \phi_k \right\rangle \right|^2 }  \\
&\leq \sum_{0 < \lambda_k \leq N_1}^{}{\left|\left\langle e^{t_0 \Delta} \frac{1}{N} \sum_{m=1}^{N}{\delta_{x_m}}, \phi_k \right\rangle \right|^2 }  + \frac{1}{N_1} \sum_{\lambda_k > N_1}^{}{\lambda_k \left|\left\langle e^{t_0 \Delta} \frac{1}{N} \sum_{m=1}^{N}{\delta_{x_m}}, \phi_k \right\rangle \right|^2 } \\
&\leq \sum_{0 < \lambda_k \leq N_1}^{}{\left|\left\langle e^{t_0 \Delta} \frac{1}{N} \sum_{m=1}^{N}{\delta_{x_m}}, \phi_k \right\rangle \right|^2 }  + \frac{1}{N_1}   \left\|  \nabla  e^{t_0 \Delta} \frac{1}{N} \sum_{m=1}^{N}{\delta_{x_m}}  \right\|^2_{L^2(M)}.
\end{align*}
Step 4 implies that this final gradient term is uniformly bounded for all $N$ and all $x_1, \dots, x_N \in \mathbb{T}$. This means, that there exists $N_1 \in \mathbb{N}$ and $\varepsilon_2 > 0$ depending only on $(M,g), t_0, \varepsilon_1$ such that, for infinitely many $N$,
$$  \sum_{0 < \lambda_k \leq N_1}^{}{\left|\left\langle e^{t_0 \Delta} \frac{1}{N} \sum_{m=1}^{N}{\delta_{x_m}}, \phi_k \right\rangle \right|^2 } \geq \varepsilon_2. $$
\textit{Step 5.} Using representation in Fourier series, we easily get that for all $s> t_0$
$$  \sum_{0 < \lambda_k \leq N_1}^{}{\left|\left\langle e^{s \Delta} \frac{1}{N} \sum_{m=1}^{N}{\delta_{x_m}}, \phi_k \right\rangle \right|^2 } \geq e^{-N_1^2 (s- t_0)} \varepsilon_2.$$
 We conclude by arguing that, for infinitely many $N \in \mathbb{N}$,
\begin{align*}
\left\langle  e^{2C \Delta}  \frac{1}{N}\sum_{m = 1}^{N}{ \delta_{x_m}},  \frac{1}{N}\sum_{m = 1}^{N}{ \delta_{x_m}} \right\rangle  &= \left\langle  e^{C \Delta}  \frac{1}{N}\sum_{m = 1}^{N}{ \delta_{x_m}}, e^{C \Delta}   \frac{1}{N}\sum_{m = 1}^{N}{ \delta_{x_m}} \right\rangle \\
&= \sum_{k=0}^{\infty}{ \left| \left\langle  e^{C \Delta}  \frac{1}{N}\sum_{m = 1}^{N}{ \delta_{x_m}}, \phi_k \right\rangle \right|^2} \\
&= \frac{1}{\vol(M)} + \sum_{k=1}^{\infty}{ \left| \left\langle  e^{C \Delta}  \frac{1}{N}\sum_{m = 1}^{N}{ \delta_{x_m}}, \phi_k \right\rangle \right|^2} \\
&\geq  \frac{1}{\vol(M)} + \sum_{k \leq N_0}^{}{ \left| \left\langle  e^{C \Delta}  \frac{1}{N}\sum_{m = 1}^{N}{ \delta_{x_m}}, \phi_k \right\rangle \right|^2} \\
&\geq  \frac{1}{\vol(M)} + e^{-N_1^2 (C - t_0)} \varepsilon_2.
\end{align*}
\end{proof}

\subsection{Proof of the Corollaries 3 and 4}
As for showing Corollaries 3 and 4, we cannot use the Gaussian because the definition of Poissonian Pair Correlation only gives pointwise control for a fixed $s$. However,
the same underlying idea can be used after replacing the Gaussian by a function $\phi:\mathbb{T} \rightarrow \mathbb{R}$ such that
\begin{enumerate}
\item $\phi$ is compactly supported around the origin and
\item $\widehat{\phi}(k) > 0$ for all $k \in \mathbb{Z}$
\end{enumerate}
Once having such a function, we may introduce, for any $t \geq 1$,
$$ \phi_t(x) = t \cdot \phi(tx).$$
Then 
$$ \int_{\mathbb{T}} \phi_t(x) dx = \mbox{constant in}~t = \int_{\mathbb{T}} \phi(x) dx.$$
Moreover, we have
\begin{align*}
  \sum_{m,n=1}^{N}{\phi_t(x_m - x_n)}  &= \left\langle \phi_t* \left( \sum_{k=1}^{N}{ \delta_{x_k}}  \right),  \sum_{k=1}^{N}{ \delta_{x_k}}  \right\rangle \\
&=  \sum_{\ell \in \mathbb{Z}}{ \widehat{\phi_t}(\ell)  \left| \sum_{k=1}^{N}{e^{-2 \pi i \ell x_k}} \right|^2} \\
&= \widehat{\phi_t}(0) N^2 + 2\sum_{\ell =1}^{\infty}{   \widehat{\phi_t}(\ell) \left| \sum_{k=1}^{N}{e^{2 \pi i \ell x_k}} \right|^2} \\
&= \left( \int_{\mathbb{T}} \phi(x) dx \right)  N^2  + 2\sum_{\ell =1}^{\infty}{   \widehat{\phi_t}(\ell) \left| \sum_{k=1}^{N}{e^{2 \pi i \ell x_k}} \right|^2}.
\end{align*}
At the same time, we can analyze the expression in a different way. Clearly,
$$  \sum_{m,n=1}^{N}{\phi_t(x_m - x_n)}  = N \cdot t \cdot \phi(0) +  \sum_{m,n=1 \atop m \neq n}^{N}{\phi_t(x_m - x_n)}.$$
Using the Poissonian Pair Correlation property, we have, for any fixed constant $c > 0$ and $t = cN$,
$$\lim_{N \rightarrow \infty} \frac{1}{N^2} \sum_{m,n = 1 \atop m \neq n}^{N}{\phi_t(x_m - x_n)} =  \int_{\mathbb{T}}{ \phi(x) dx}.$$
Combining all these inequalities, we see that for any $c>0$ and $t= cN$
$$ \limsup_{N \rightarrow \infty} \sum_{\ell =1}^{\infty}{   \widehat{\phi_t}(\ell) \left| \frac{1}{N} \sum_{k=1}^{N}{e^{2 \pi i \ell x_k}} \right|^2} \leq c \cdot \phi(0).$$
Therefore, we have, for any fixed $\ell \in \mathbb{Z}$ that
$$ \limsup_{N \rightarrow \infty} \widehat{\phi_t}(\ell) \left| \frac{1}{N} \sum_{k=1}^{N}{e^{2 \pi i \ell x_k}} \right|^2 \leq c \cdot \phi(0).$$
It is easy to see that, for $\ell$ fixed,
$$ \widehat{\phi_t}(\ell)  \rightarrow \int_{\mathbb{T}}{\phi(x) dx}$$
and we note that this quantity does not vanish since it equals $\widehat{\phi}(0) > 0$.
Since $c$ can be chosen arbitrarily small, equidistribution follows from Weyl's criterion. The same argument applies to the fractional case, it is
merely a different choice of $t$. This argument has other implications that have been developed elsewhere \cite{steinjnt}.

\section{Proof of the Corollary 5}
\begin{lem} Let $t > 0$ and consider $\theta_t:\mathbb{T} \rightarrow \mathbb{R}_{+}$ given by.
  $$ \theta_t(x) = 1 + 2 \sum_{n=1}^{\infty}{e^{-4 \pi^2 n^2 t}\cos{(2 \pi n x)}}.$$
If $\varepsilon > 0$ and
$$ x \geq 2\sqrt{\log{(2/\varepsilon)}} \sqrt{t},~ \mbox{then} \qquad \int_{-x}^{x}{\theta_t(y) dy} \geq 1 - \varepsilon.$$
\end{lem}
\begin{proof} A simple topological argument allows to compare the heat kernel on the torus with the heat kernel on the real line: on the torus we have
the possibility of looping around which we do not have on the real line. Therefore, for all $x \in \mathbb{R}$ and all $t>0$
$$ 1 + 2 \sum_{n=1}^{\infty}{e^{-4 \pi^2 n^2 t}\cos{(2 \pi n x)}} \geq \frac{1}{\sqrt{4\pi t}} e^{-\frac{x^2}{4t}}.$$
The result then follows from the Chernoff bound
 $$ \int_{x}^{\infty}{ \frac{1}{\sqrt{2\pi }} e^{-\frac{y^2}{2}} dy} \leq e^{-\frac{x^2}{2}}.$$
\end{proof}

\begin{proof}[Proof of Corollary 5]
Let $x_1, \dots, x_N \in \mathbb{T}$ be given and assume that
$$ D_N(\left\{ x_1, \dots, x_N \right\}) = \varepsilon.$$
Then there exists an interval $J \subset \mathbb{T}$ such that
 $$ \left| \frac{  \# \left\{1 \leq m \leq N: x_m \in J\right\}}{N} - |J| \right| = \varepsilon.$$
 We distinguish two cases:
 $$  \frac{  \# \left\{1 \leq m \leq N: x_m \in J\right\}}{N} = |J| + \varepsilon \qquad \mbox{and} \qquad  \frac{  \# \left\{1 \leq m \leq N: x_m \in J\right\}}{N} = |J| - \varepsilon.$$
 We start with the first case, the second
case is essentially identical. 
We set 
$$ t= \frac{1}{100}\frac{\varepsilon^2}{\log{(20/\varepsilon)}}.$$
This choice guarantees, using the Lemma above, that
$$ \int_{-\varepsilon/4}^{\varepsilon/4}{\theta_t(y) dy} \geq 1- \frac{\varepsilon}{10}.$$
We will now consider the slightly larger interval $J^*$ given as the $\varepsilon/4-$neighborhood of $J$.
We see that, for infinitely many $N \in \mathbb{N}$,
\begin{align*}
 \frac{1}{|J^*|} \left\|    \left[e^{(t/2) \Delta} \frac{1}{N} \sum_{m=1}^{N}{\delta_{x_m}} \right](x) \right\|_{L^1(J^*)} &\geq \left(1-\frac{\varepsilon}{10}\right)  \frac{1}{|J^*|}\left\|    \frac{1}{N} \sum_{m=1}^{N}{\delta_{x_m}} \right\|_{L^1(J)} \\
 &\geq \left(1-\frac{\varepsilon}{10}\right)   \frac{1}{|J| + \varepsilon/2} \left\|    \frac{1}{N} \sum_{m=1}^{N}{\delta_{x_m}} \right\|_{L^1(J)}\\
&\geq  \left(1-\frac{\varepsilon}{10}\right)  \frac{|J| + \varepsilon}{|J| + \varepsilon/2}\\
&\geq  \left(1-\frac{\varepsilon}{10}\right)\frac{1+ \varepsilon}{1+ \varepsilon/2}\\
&\geq 1 + \frac{\varepsilon}{10}.
\end{align*}
We use the Cauchy-Schwarz inequality in the form
$$ \left( \int_{J^*}{ f(x) - 1~ dx} \right)^2 \leq  \left( \int_{J^*}{ (f(x) - 1)^2 dx} \right) |J^*|^{} \leq \left( \int_{\mathbb{T}}{ (f(x) - 1)^2 dx} \right) $$
to conclude that
$$  \left\|    \left[e^{(t/2) \Delta} \frac{1}{N} \sum_{m=1}^{N}{\delta_{x_m}} \right](x) -1 \right\|^2_{L^2(\mathbb{T})}  \geq \frac{\varepsilon^2}{100}.$$
\begin{center}
\begin{figure}[h!]
\begin{tikzpicture}[scale=0.8]
\draw [ultra thick] (0,0) -- (15,0);
\filldraw (1,0) circle (0.08cm);
\filldraw (3,0) circle (0.08cm);
\filldraw (6,0) circle (0.08cm);
\filldraw (8,0) circle (0.08cm);
\filldraw (7,0) circle (0.08cm);
\filldraw (9,0) circle (0.08cm);
\filldraw (13,0) circle (0.08cm);
\filldraw (14,0) circle (0.08cm);
 \draw[scale=1,thick, samples=200,domain=0:15,smooth,variable=\x]  plot ({\x},{3*exp(-7*(\x-6)*(\x-6))});
 \draw[scale=1,thick, samples=200,domain=0:15,smooth,variable=\x]  plot ({\x},{3*exp(-7*(\x-8)*(\x-8))});
 \draw[scale=1,thick, samples=200,domain=0:15,smooth,variable=\x]  plot ({\x},{3*exp(-7*(\x-7)*(\x-7))});
 \draw[scale=1,thick, samples=200,domain=0:15,smooth,variable=\x]  plot ({\x},{3*exp(-7*(\x-9)*(\x-9))});
\draw [ultra thick] (6,-0.2)--(6,0.2);
\draw [ultra thick] (9,-0.2)--(9,0.2);
\draw [ultra thick] (5.6,-0.2) to[out=130, in=230] (5.6,0.2);
\draw [ultra thick] (9.4,-0.2) to[out=50, in=310] (9.4,0.2);
\end{tikzpicture}
\caption{Slightly too many points in an interval (bounded by straight lines) implies slightly too much $L^1-$mass of the heat
kernel in a slightly larger interval (bounded by curved lines).}
\end{figure}
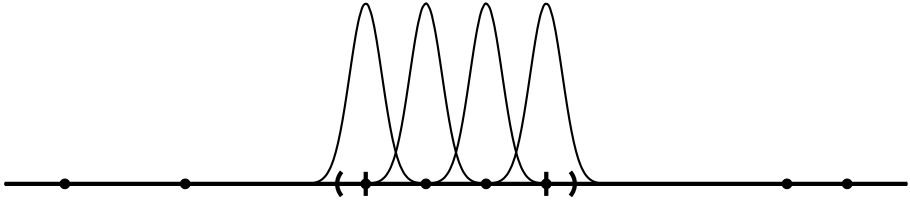
\end{center}

An explicit computation shows that
\begin{align*}
 \left\|    \left[e^{(t/2) \Delta} \frac{1}{N} \sum_{m=1}^{N}{\delta_{x_m}} \right](x) -1 \right\|^2_{L^2(\mathbb{T})} &= \int_{\mathbb{T}}{   \left[e^{(t/2) \Delta} \frac{1}{N} \sum_{m=1}^{N}{\delta_{x_m}} \right](x)^2 dx} -1 \\
&= \left\langle  e^{(t/2) \Delta} \frac{1}{N} \sum_{m=1}^{N}{\delta_{x_m}} ,  e^{(t/2) \Delta} \frac{1}{N} \sum_{m=1}^{N}{\delta_{x_m}}  \right\rangle - 1\\
&=  \left\langle  e^{t \Delta} \frac{1}{N} \sum_{m=1}^{N}{\delta_{x_m}} , \frac{1}{N} \sum_{m=1}^{N}{\delta_{x_m}}\right\rangle - 1\\
&=  -1 + \frac{1}{N^2} \sum_{m,n = 1}^{N}{\theta_{t}(x_n - x_m)} 
\end{align*}
and we can conclude the result. The other case of not enough points follows analogously except that $J^*$ is obtained from shrinking $J$.
\end{proof}

\textbf{Acknowledgment.} The author is grateful to Christoph Aistleitner for several discussions and to Jens Marklof for pointing out an error
in the original proof of Corollary 3 and 4.

\end{document}